\documentclass[twoside, paper=a4, fontsize=10pt]{article} 
\usepackage[colorlinks=true,linkcolor=blue,citecolor=blue,urlcolor=blue]{hyperref}
\usepackage[T1]{fontenc}
\usepackage[utf8]{inputenc}
\usepackage[english]{babel}

\usepackage[bbgreekl]{mathbbol}
\usepackage{amsmath,amssymb,amsfonts,amsthm}
\usepackage{mathrsfs}

\usepackage[charter]{mathdesign}

\usepackage[pdftex]{graphicx}	
\usepackage{url}

\usepackage{titlesec}
\titleformat{\section}[hang]{\Large \normalfont \filcenter \scshape}{\thesection}{10pt}{}
\titleformat{\subsection}[hang]{\normalfont \large \bfseries}{\thesubsection}{6pt}{}

\titlespacing*{\section}{0pt}{14pt}{3pt}
\titlespacing*{\subsection}{0pt}{14pt}{3pt}
\titlespacing*{\paragraph}{0pt}{7pt}{8pt}

\usepackage{enumitem}

\usepackage{fancyhdr}
\usepackage[titletoc]{appendix}
\AtBeginDocument{}
\appendixtocon

\usepackage[numbers,comma,sort]{natbib}
\usepackage{amsmath}
\usepackage{caption}	

\usepackage[affil-it]{authblk}

\numberwithin{equation}{section}

\newtheorem{theorem}{Theorem}[section]
\newtheorem{mainTh}{Theorem}[]
\newtheorem{lemma}[theorem]{Lemma}

\newtheorem{definition}[theorem]{Definition}
\newtheorem{proposition}[theorem]{Proposition}
\newtheorem{remark}[theorem]{Remark}

\begin{document}

\newcommand{\A}{\mathcal{A}}
\newcommand{\C}{\mathcal{C}}
\newcommand{\R}{\mathbb{R}}
\newcommand{\Q}{\mathbb{Q}}
\newcommand{\Z}{\mathbb{Z}}
\newcommand{\N}{\mathbb{N}}
\newcommand{\levy}{\mathscr{B}}
\newcommand{\sheet}{\mathbbm{B}}
\newcommand{\BB}{\boldsymbol{B}}
\newcommand{\alphar}{\texttt{$\boldsymbol{\alpha}$}}
\newcommand{\sigmar}{\texttt{\large $\boldsymbol{\sigma}$}}
\newcommand{\qA}{q_{\underline{\mathcal{A}}}}
\newcommand{\EE}{\mathbb{E}}
\newcommand{\PP}{\mathbb{P}}
\newcommand{\VV}{\mathbb{V}\text{ar}}
\newcommand{\h}{\normalfont{\textit{\textbf{h}}}}
\newcommand{\dH}{\normalfont{\text{dim}}_H}
\newcommand{\F}{\mathbf{F}}
\newcommand{\x}{\mathtt{\mathbf{x}}}
\newcommand{\psiu}{\Psi^{(u)}} 
\newcommand{\psil}{\Psi^{(\ell)}}
\newcommand{\psiut}{\widetilde{\Psi}^{(u)}} 
\newcommand{\psilt}{\widetilde{\Psi}^{(\ell)}}
\newcommand{\etau}{\eta^{(u)}} 
\newcommand{\etal}{\eta^{(\ell)}}

\title{\vspace{-1cm} \normalfont \Large \uppercase {Some singular sample path properties of a multiparameter fractional Brownian motion}}

\author{Alexandre Richard}
\affil{\small \'Ecole Centrale Paris-INRIA Regularity Team and Bar-Ilan University\\  Grande Voie des Vignes, 92295 Châtenay-Malabry, France\\ e-mail: \href{mailto:alexandre.richard@inria.fr}{alexandre.richard@inria.fr}}

\date{}

\maketitle

\begin{abstract}
We obtain a spectral representation and compute the small ball probabilities for a (non-increment stationary) multiparameter extension of the fractional Brownian motion. We derive from these results a Chung-type law of the iterated logarithm at the origin, and exhibit the singular behaviour of this multiparameter fractional Brownian motion, as it behaves very differently at the origin and away from the axes. A functional version of this Chung-type law is also provided.
\end{abstract}

{\sl AMS classification\/}: 60\,F\,17, 60\,G\,60, 60\,G\,17, 60\,G\,15, 60\,G\,22, 28\,C\,20.

{\sl Key words\/}: Fractional Brownian motion, Gaussian random fields, small deviations, spectral representation, Chung's law of the iterated logarithm.

\section{Introduction}

In the 1920's, Khinchine introduced for the first time a law of the iterated logarithm for sums of independent and identically distributed random variables. Thereafter, many works extended this result and in particular \citet{Chung} presented a new law of the iterated logarithm for Brownian motion of the \emph{lim inf} type, thus capturing the slowest local oscillations. This law was generalized over the last few decades in numerous ways to Gaussian \cite{Monrad, KhoshnevisanShi98, tudorxiao07} and non-Gaussian processes \cite{KhoshnevisanLewis, DPLV, CsakiHu}, Gaussian samples \cite{KuelbsLiTalagrand} and empirical processes \cite{Deheuvels}, and Gaussian random fields \cite{Talagrand94,MasonShi,LuanXiao2010}, for a non-exhaustive list of works. Two key steps in establishing Chung-type laws of the iterated logarithm (abbreviated as LIL) are usually to determine small ball probabilities, together with a good decomposition of the process into independent processes. Because of the difficulty to obtain small ball probabilities, finding a Chung-type law is generally more difficult than finding a standard LIL, as can be seen from the example of the fractional Brownian sheet, for which a LIL is given in \cite{Wang07}, but a Chung-type law is far from clear, except in special cases \cite{Talagrand94,MasonShi, BelinskyLinde}.

We propose to study a process which is both a natural extension of the fractional Brownian motion (fBm for short) into a multiparameter process, and an extension of the Brownian sheet into a fractional process. This multiparameter fractional Brownian motion, denoted by $\BB$, is a centred Gaussian process on $\R_+^\nu$, $\nu\in \N^*$, with covariance defined for some Hurst parameter $h\in (0,1/2]$ by:
\begin{equation}\label{eq:cov}
k^{(\nu)}(s,t) = \frac{1}{2}\left(\lambda([0,s])^{2h} + \lambda([0,t])^{2h} - \lambda([0,s]\bigtriangleup [0,t])^{2h}\right)\ ,\ s,t\in\R_+^\nu \ ,
\end{equation}
where $\lambda$ denotes the Lebesgue measure in $\R^\nu$, $[0,t]$ is the rectangle with vertices at $0$ and $t$, and $\bigtriangleup$ is the symmetric difference of sets. This is a special case of a family of covariance on sets introduced by \citet{ehem} to define the set-indexed fractional Brownian motion, and extended in \cite{Richard} to a more general expression as a covariance on $L^2(T,m)$. 
This process differs from the other extensions that are the Lévy fractional Brownian motion, and the fractional Brownian sheet, although it shares several properties with them (see \cite{Richard} for a more thorough discussion on the links between these processes). Besides, $h=1/2$ in (\ref{eq:cov}) yields the usual Brownian sheet. However, the results presented here hold for $h<1/2$ and cannot be extended to $h=1/2$.

Let us make a couple of comments on this process that will help understand the behaviour of its sample paths. \emph{First}, this multiparameter fBm is not increment stationary (unlike processes in \citet{Monrad, Talagrand95, Xiao96} and for a more general theory, \citet{LuanXiao2010} where this is an assumption), which has the following immediate consequences: there is no Ito-Yaglom spectral representation \cite{Yaglom57}, nor any obvious decomposition as a sum of independent processes (as for the Brownian sheet \cite{Talagrand94}), so one of the key steps mentioned in the first paragraph to obtain a Chung-type LIL is missing; and even if a local property such as a LIL can be proven at some point $t_0\in \R_+^\nu$, it cannot be automatically extended to any other point. Note that the same happens for the fractional Brownian sheet: its LIL is known away from $0$ (see \cite{mwx}), but not in the neighbourhood of the origin (except for particular increments as in \cite{Wang07}).
\emph{Second}, the multiparameter fBm behaves very differently on the axes and away from them. On a domain of $\R_+^\nu$ that does not approach the axes, the distance induced by $\BB$, defined by $d_{\BB}(s,t)= \lambda([0,s]\bigtriangleup [0,t])^h$, is equivalent to $d_{X}(s,t) = \|s-t\|^h$, see Lemma \ref{lem:compDist}. The latter is in fact the distance induced by the L\'evy fractional Brownian motion $X$ (studied for example in \cite{Talagrand95}), thus it is expected that these processes will share certain sample path properties, at least away from the axes. This is the purpose of the work of \citet{HerbinXiao}, where the authors propose a modulus of continuity, a law of the iterated logarithm and compute the Hausdorff dimension of the level sets of $\BB$. These results coincide with their analogue for the Lévy fBm $X$, but for the Chung LIL, which is a local result, this is only true away from the axes. This difference originates from the small ball probabilities $\PP\left(\sup_{t\in B(t_0,r)} |\BB_t|\leq \epsilon\right)$, since these quantities differ significantly when $t_0$ is on the axes (Equation (\ref{eq:smallBalls})) or not (Equation (\ref{eq:smallBalls0})), and justifies this notion of singularity at the origin.

As suggested in the previous paragraph, the main new contributions here are a sharp estimate of the small ball probabilities of the multiparameter fBm, and a spectral representation for a large class of $L^2$-indexed Gaussian processes. Given a measure space $(T,m)$ and a mapping $s\mapsto \varphi_s$ defined on a subset of $\R^\nu$ with values in $L^2(T,m)$, a centred Gaussian process in this class is given by the covariance $(s,t)\mapsto \frac{1}{2}\left(\|\varphi_s\|_{L^2}^{4h} + \|\varphi_t\|_{L^2}^{4h} - \|\varphi_s-\varphi_t\|_{L^2}^{4h}\right)$, $h\in (0,1/2]$. Examples of such processes include the multiparameter fBm $\BB$, the Lévy fBm $X$ (see Centsov's construction \cite[Chapter 8]{Samorodnitsky} and \cite{Richard}), but also non-isotropic processes which are variants of the fractional Brownian sheet, etc. The rest of the proofs uses standard techniques, which points to the fact that in order to obtain a Chung-type LIL for any other process in this class of Gaussian processes, it would remain to determine its small ball probabilities and follow the steps exhibited here.

\paragraph{Statement of the main results.}
The spectral representation we obtain is related to stable measures in Banach spaces: we prove in Proposition \ref{prop:radialDecMpfBm} that for any separable Hilbert space $H$, there exists an abstract Wiener space $(H,E,\mu)$ such that for any $h\in(0,1/2)$, there is a strictly stable measure $\Gamma$ on $E$ whose characteristic function is given by $\exp\{-\|S\xi\|_H^{4h}/2\},\ \xi\in E^*$, where $S$ is a map defined in (\ref{eq:defS}). Besides, this measure has a Lévy-Khintchine decomposition, with Lévy measure $\Delta$. This leads to a spectral representation of the multiparameter fBm, which is our first main result: 
\begin{mainTh}\label{prop:spectralRep}
Let $h\in(0,1/2)$ and $\Delta$ be the Lévy measure of Proposition \ref{prop:radialDecMpfBm}, based on the Hilbert space $H=L^2\left([0,1]^\nu,\lambda\right)$. Let $\mathbb{B}$ be a complex Gaussian white noise on the Borel sets of $E$, with control measure $\Delta$, and define the stochastic process $\{\mathscr{B}(\xi),\ \xi\in E^*\}$ by:
\begin{equation*}
\mathscr{B}(\xi) = \int_E \left(1-e^{i\langle \xi,x \rangle}\right)\ \textrm{d}\mathbb{B}_x \ , \ \xi\in E^*\ .
\end{equation*}
Then, the definition of $\mathscr{B}$ extends to $H$, and the process $\left\{\mathscr{B}\left(\mathbf{1}_{[0,t]}\right),\ t\in[0,1]^\nu\right\}$ is a centred Gaussian process with covariance (\ref{eq:cov}), i.e. it is a multiparameter fractional Brownian motion.
\end{mainTh}

The idea of introducing an infinite-dimensional process to derive properties on $\BB$ will be exploited again for the small deviations. Indeed, a local nondeterminism result established in \cite{Richard} for a large class of Gaussian processes, permits to relate the small ball probabilities of $\BB$ to the metric entropy of small balls as measured by $d_{\BB}$. Thus, away from the axes, the following holds (see more details in Remark \ref{rem:smallBalls}), for $r>0$ and $t_0$ such that $B(t_0,r)\subset (0,\infty)^\nu$, as $\varepsilon \rightarrow 0$:
\begin{align}\label{eq:smallBalls0}
-\log\PP\left(\sup_{t\in B(t_0,r)} |\BB_t|\leq \varepsilon\right)\asymp \left(\frac{r}{\varepsilon^{1/h}}\right)^\nu \ .
\end{align}
This is very different from the result for $t_0=0$:
\begin{mainTh}\label{th:smallBalls}
For $h<1/2$, there are constants $\kappa_1>0$ and $\kappa_2>0$ such that for any fixed $r\in (0,1)$ and any $\varepsilon$ small enough (compared to $r$),
\begin{align}\label{eq:smallBalls}
\exp\left\{-\kappa_2 \frac{r^{\nu^2}}{\varepsilon^{\nu/h}}\right\} \leq \PP\left(\sup_{t\in [0,r]^\nu} |\BB_t|\leq \varepsilon \right) \leq \exp\left\{ -\kappa_1 \frac{r^{\nu^2}}{\varepsilon^{\nu/h}}\right\}\ .
\end{align}
\end{mainTh}

\noindent Theorems \ref{prop:spectralRep} and \ref{th:smallBalls} provide the two key ingredients to prove a lower and an upper bound in a Chung-type LIL. The modulus in the lower bound will be: \[\psil(r) = r^{\nu h}\ \psilt(r) =  r^{\nu h}\ (\log\log r^{-1})^{-h/\nu}\ ,\]while the modulus for the upper bound will be $\psiu(r) = r^{\nu h}\ \psiut(r)$, where $\psiut$ is an increasing function started at $0$ and such that $\psiut \geq \psilt$, so that in particular, $\psiut(r)^{-1} = o(r^{-\nu h})$ as $r\rightarrow 0$. The existence of $\psiut$ is proven in Section \ref{sec:ChungLIL}, and related implicitly to the following decay function of $\Delta$:
\begin{equation}\label{eq:defF}
\F(\x) = \sup_{\varphi \in A(1)} \int_{\|x\|_E < \x } \left(1-\cos\langle\mathcal{I}(\varphi),x\rangle\right) \ \Delta(\text{d}x) \ ,
\end{equation}
where $A(1)$ is a compact subset of $H$, defined in the sequel. Note that $\F$ may depend on $\nu$, hence so does $\psiut$. Let us finally define $M(r) = \sup_{t\in[0,r]^\nu} |\BB_t|, \ r\in[0,1]$.
\begin{mainTh}\label{th:LIL}
Let $h\in (0,1/2)$ and let $M$, $\psil$ and $\psiu$ be as above. Then we have almost surely:
\begin{equation*}
\liminf_{r\rightarrow 0^+} \frac{M(r)}{r^{\nu h} \psilt(r)} \geq \kappa_1^{h/\nu} \ \text{ and }\ \liminf_{r\rightarrow 0^+} \frac{M(r)}{r^{\nu h} \psiut(r)} \leq \kappa_2^{h/\nu} \ ,
\end{equation*}
where $0<\kappa_1 \leq \kappa_2<\infty$ are the constants appearing in the small deviations (see Equation (\ref{eq:smallBalls})).
\end{mainTh}
\noindent This result is not sharp \emph{a priori}, and depends on the rate of decay of $\F$. We discuss how this gap could be filled at the end of Section \ref{sec:ChungLIL}. In Remark \ref{rem:diffLIL}, we comment on the leading term $r^{\nu h}$, compared to $r^h$ for the Chung LIL of $\BB$ away from $0$. This is a direct consequence of the difference in the small ball probabilities (\ref{eq:smallBalls0}) and (\ref{eq:smallBalls}).

\vspace{0.2cm}

In \citet{Strassen(1964)}, while looking for an invariance principle for scaled random walks, the author obtained the fact that the same scaling on a Brownian motion gives a family of processes which is almost surely relatively compact in the unit ball of $H^1_0$, the Sobolev space of continuous functions started at $0$ with square-integrable weak derivative. Functional laws of the iterated logarithm have now been widely studied in the literature: \citet{Csaki} was the first to get a rate of convergence for certain functions in this unit ball, and this result was extended by \citet{deAcosta83} to scaled random walks, for any function of the unit ball of the RKHS (with radius strictly smaller than $1$). After several contributions, \citet{KuelbsLiTalagrand} finally brought a new understanding of the rate of convergence towards the unit sphere in the general frame of Gaussian samples in Banach spaces. Similarly to the standard LIL, the functional result for fractional Brownian motion was also given by \citet{Monrad}. So for the multiparameter fBm, let us define, for $r\in (0,1)$,
\begin{equation*}
\etal_r(t) = \frac{\BB(rt)}{r^{\nu h}\sqrt{\log\log(r^{-1})}} \ , \forall t\in[0,1]^\nu 
\end{equation*}
and
\begin{equation*}
\etau_r(t) = \frac{\BB(rt)}{r^{\nu h} \left(\psiut(r)\right)^{-\nu/2h}} \ , \forall t\in[0,1]^\nu 
\end{equation*}
the lower and upper rescaled multiparameter fBm for which we seek an invariance principle. 

\begin{mainTh}\label{th:functLIL}
Let $h\in(0,1/2)$ and let $H^\nu$ denote the reproducing kernel Hilbert space of $k^{(\nu)}$ (defined in (\ref{eq:cov})). Let $\varphi\in H^\nu$ being of norm strictly smaller than $1$. Then, there exist two positive and finite constants $\gamma^{(\ell)}(\varphi)$ and $\gamma^{(u)}(\varphi)$ such that, almost surely,
\begin{align*}
&\liminf_{r\rightarrow 0^+} \ \psilt(r)^{-1-\nu/2h} \sup_{t\in[0,1]^\nu}|\etal_r(t) - \varphi(t)| \geq \gamma^{(\ell)}(\varphi) \\
&\liminf_{r\rightarrow 0^+} \ \psiut(r)^{-1-\nu/2h} \sup_{t\in[0,1]^\nu}|\etau_r(t) - \varphi(t)| \leq \gamma^{(u)}(\varphi)\ .
\end{align*}

\end{mainTh}
\noindent As usual, taking $\varphi=0$ yields the standard law of the iterated logarithm.

\vspace{0.2cm}

\paragraph{Organization of the paper.} In section \ref{sec:preliminary}, we prove some preliminary results. The main new tools and ideas essentially lie in this section. Some facts about Wiener spaces and infinite-dimensional stable measures are recalled, and then we prove the spectral representation of $\BB$ (Theorem \ref{prop:spectralRep}) and the small deviations estimate of Theorem \ref{th:smallBalls}. We prove Theorem \ref{th:LIL} in Section \ref{sec:ChungLIL} and Theorem \ref{th:functLIL} in Section \ref{sec:FLIL}. Finally, we make some conclusive remarks on the Hausdorff dimension of the graph of $\BB$ and its local Hölder regularity at the origin.

\section{Preliminaries}\label{sec:preliminary}

We recall a few notions about Gaussian measures on Banach spaces and abstract Wiener spaces (see also \cite[Lemma 2.1]{Kuelbs76}). Let $E$ be a separable Banach space and $\mu$ a Gaussian measure on $E$. Let $H$ be the completion of $E^*$ by the action of the covariance operator $S$ of $\mu$, defined by:
\begin{equation}\label{eq:defS}
S\xi = \int_E x\ \langle \xi,x\rangle\ \mu(\text{d}x) \ , \xi\in E^*\ ,
\end{equation}
which maps $E^*$ into a subspace of $E$, and the completion is with respect to the scalar product:
\begin{equation*}
\left(S\xi, S\xi'\right)_H = \int_E \langle\xi,x\rangle\ \langle\xi',x\rangle\ \mu(\text{d} x) \ .
\end{equation*}
This permits to define a sequence $\{\xi_n, n\in \N\}$ in $E^*$, such that $\{S\xi_n, n\in \N\}$ is a complete orthonormal system (CONS) in $H$. We recall that the Paley-Wiener map $\mathcal{I}$ is defined as the isometric extension of the map ${\xi\in E^* \mapsto \langle \xi,\cdot \rangle}$ to a map from $H$ to $L^2(\mu)$. Conversely, it is also possible to start from a separable Hilbert space and to construct an embedding into a larger Banach space, on which there exists a Gaussian measure whose covariance will be related to the inner product on $H$. This is the abstract Wiener space (AWS) approach \cite{Gross,Stroock}. 

\begin{definition}[Reproducing Kernel Hilbert Space]\label{def:RKHS}
Let $(T,m)$ be a separable and complete metric space and $R$ a continuous covariance function on $T\times T$. $R$ determines a unique Hilbert space $H(R)$ satisfying the following properties: \emph{i)} $H(R)$ is a space of functions mapping $T$ to $\R$; \emph{ii)} for all $t\in T$, $R(\cdot,t) \in H(R)$; \emph{iii)} for all $t\in T$, $\forall f\in H(R)$, $\left(f,R(\cdot,t)\right)_{H(R)} = f(t)$ .
\end{definition}

In \cite{Richard}, starting from the reproducing kernel Hilbert space (RKHS) of the $L^2([0,1]^\nu,\lambda)$-indexed fBm of parameter $h\in (0,1/2]$, built upon the kernel $k(f,g) = 1/2 \left(\lambda(f^2)^{2h} + \lambda(g^2)^{2h} - \lambda((f-g)^2)^{2h}\right)$ for $f,g\in L^2([0,1]^\nu,\lambda)$ and denoted by $H(k)$, an integral representation for the multiparameter fBm was obtained. Indeed, first define:
\begin{equation}\label{eq:L2fBm}
\mathcal{W}(\varphi) = \int_E \langle\mathcal{I}(\varphi), x\rangle \ \text{d}\mathbb{W}_x\ ,
\end{equation}
for $\varphi \in H(k)$, where $\mathbb{W}$ is a white noise on some Gaussian measure space $(E,\mu_h)$ with RKHS $H(k)$. Then, $\mathbf{W}_t = \mathcal{W}(k(\mathbf{1}_{[0,t]},\cdot))$ is a multiparameter fBm.
Note the link with $k^{(\nu)}$ defined in (\ref{eq:cov}), as for any $s,t\in[0,1]^\nu$, $k(\mathbf{1}_{[0,s]},\mathbf{1}_{[0,t]}) = k^{(\nu)}(s,t)$.

\vspace{0.1cm}

In general, the embedding between $H$ and $E$ is continuous. We will need it to be Hilbert-Schmidt for an extension of Bochner's theorem to be valid. The following lemma states that starting from a separable Hilbert space $H$, it is possible to find $E$ and $\mu$ satisfying this property and such that $(H,E,\mu)$ is an AWS.

\begin{lemma}\label{lem:HS}
Let $H$ be a separable Hilbert space. There are a separable Hilbert space $(E,\|\cdot\|)$ and a Gaussian measure $\mu$ on $E$ such that $(H,E,\mu)$ is an abstract Wiener space and the embedding $H\subset E$ is Hilbert-Schmidt.
\end{lemma}

\begin{proof}
Let us assume that there exists separable Hilbert spaces $H_0$ and $E_0$ such that $H_0$ is densely embedded into $E_0$ by an operator $R$ which is Hilbert-Schmidt, and that there exists a Gaussian measure $\mu_0$ such that $(H_0,E_0,\mu_0)$ is an abstract Wiener space. In that case, $R$ is the covariance operator. Let $u$ be any linear isometry between $H_0$ and $H$, and denote by $(H,E,\mu)$ the AWS given by $E=\tilde{u}(E_0)$ and ${\mu = \tilde{u}_*\mu_0}$, where $\tilde{u}$ is the isometric extension of $u$ (see \cite[p.317]{Stroock}). Since $E_0$ is a Hilbert space, $E$ is also a Hilbert space and the operator ${R' = \tilde{u}\circ R\circ u^{-1}}$ is the natural embedding from $H$ into $E$, and is of Hilbert-Schmidt type.

\noindent The existence of such a $(H_0,E_0,\mu_0)$ triple follows either from examples as in sections 6 and 7 of \cite{KuelbsLiTalagrand}, or by the construction of the next paragraph.
\end{proof}

\noindent Let us detail the Wiener space structure of $(H,E,\mu)$ when $E$ is a Hilbert space. Let $\{x_n, n\in \N\}$ be a complete orthonormal system of $\left(E,(\cdot,\cdot)_E\right)$. For each $n$, let $\lambda_n^2$ be the variance of $(x_n,\cdot)_E \in E^*$ under $\mu$. Note that $\sum_{n\geq 1} \lambda_n^2<\infty$, which follows from the fact that:
\begin{align*}
\sum_{n\in\N} \lambda_n^2 &= \sum_{n\in\N} \int_E (x_n,x)_E^2\ \mu(\text{d}x) \\
&= \int_E \|x\|_E^2 \ \mu(\text{d}x) \ ,
\end{align*}
and this quantity is finite (we know from \citet{Fernique1970} that $\mu$ has exponential moments). Then $H$ is given by:
\begin{equation}\label{repEHilbert}
H = \left\{x\in E: \sum_{n=1}^\infty \left(\frac{(x,x_n)_E}{\lambda_n}\right)^2 <\infty \right\} \ .
\end{equation}
$\{h_n = \lambda_n x_n,\ n\in\N\}$ defines a CONS of $H$ for the scalar product given by $(x,h_n)_H=\lambda_n^{-1} (x,x_n)_E$, for any $x\in H$, and any $n\in \N$. Then, one can check that $H$ is densely and continuously embedded into $E$.

\subsection{Spectral representation of the multiparameter fBm}

The multiparameter fBm does not have independent increments, hence there is no spectral measure in the sense of \citet{Yaglom57}. Such a situation already appeared for the bifractional Brownian motion \cite{tudorxiao07}, but the difficulty was overcome due to the equivalence of the distance induced by the bifractional Brownian motion with the Euclidean distance, even near $0$. This is no longer true here. We shall use instead the $L^2$-increment stationarity in the Wiener space, to produce independent processes. 

\vspace{0.1cm}

Now, we address the spectral decomposition itself. For any $\alpha\in (0,2]$, the application $\xi\in E^* \mapsto \|S \xi\|_H^{\alpha}$ is continuous (because of the inequality ${\|\cdot\|_H \leq C\ \|\cdot\|_{E^*}}$) and negative definite (by an argument on Bernstein functions, see for instance the introduction of \cite{Richard}). Thus, according to Schoenberg's theorem, $\xi\mapsto \exp(-t\|S \xi\|_H^{\alpha})$ is positive definite for any $t\in \R_+^*$. It follows from Lemma \ref{lem:HS} and Sazonov's theorem, according to which a Hilbert-Schmidt map is $\gamma$-radonifying\footnote{see for instance \cite{Yan} for Sazonov's theorem, and \cite{Rockner2011} for its use in a similar context, as well as the references therein.}, that since $\xi\mapsto \exp(-\frac{1}{2}\|S \xi\|_H^{\alpha})$ is continuous on $H$, it is the Fourier transform of a measure $\Gamma_\alpha$ on $E$, i.e:
\begin{equation*}
e^{-\frac{1}{2}\|S \xi\|_H^{\alpha}} = \int_E e^{i\langle \xi, x\rangle}\ \text{d}\Gamma_\alpha(x) \ .
\end{equation*}

\noindent The measure $\Gamma_\alpha$ is a strictly stable and symmetric measure on $E$ of index $\alpha$, since it satisfies (we denote by $\widehat{\Gamma}_\alpha$ the Fourier transform of $\Gamma_\alpha$), for any integer $k$, and any $\xi\in E^*$:
\begin{equation*}
\left(\widehat{\Gamma}_\alpha(\xi)\right)^k = \widehat{\Gamma}_\alpha(k^{1/\alpha}\xi) \quad \textrm{ and }\quad \widehat{\Gamma}_\alpha(-\xi) = \widehat{\Gamma}_\alpha(\xi) \ .
\end{equation*}
In particular, we see that $\Gamma_\alpha$ is infinitely divisible.
\citet{Kuelbs73} extended the spectral decomposition of $\alpha$-stable measures on $\R$ to the Hilbert space setting. Thus, when $\alpha\in (0,2)$,  $\Gamma_\alpha$ has a L\'evy measure $\Delta_\alpha$ and can be written:
\begin{equation*}
\int_E e^{i\langle \xi, x\rangle}\ \text{d}\Gamma_\alpha(x) = \exp\left\{ \int_E \left(e^{i\langle\xi,x \rangle} -1 -i\frac{\langle\xi,x\rangle}{1+\|x\|_E}\right)  \ \Delta_\alpha(\text{d}x)\right\}\ ,
\end{equation*}
with $\Delta_\alpha$ satisfying $\int_E (1\wedge \|x\|_E^2)\ \Delta_\alpha(\text{d}x)<\infty$ and $\Delta_\alpha(\{0\})=0$. That $\alpha$ is strictly smaller than $2$ is essential, and this will be assumed implicitly throughout the rest of this article. It follows, cancelling the imaginary part (by symmetry of $\Delta_\alpha$), that:
\begin{equation}\label{eq:specRep}
\forall \xi \in E^*, \quad\quad -\|S\xi\|_H^{\alpha} = 2 \int_E \left(\cos\langle\xi, x \rangle - 1\right)\ \Delta_\alpha(\text{d}x)
\end{equation}

In the finite-dimensional setup, $\Delta_\alpha$ is known explicitly and appears in the spectral representation of the L\'evy fractional Brownian motion, as in \cite{Talagrand95}. In fact, Lemma 2.1 and 2.2 of \cite{Kuelbs73} give a radial decomposition of $\Delta_\alpha$ in terms of a finite measure $\sigma_\alpha$ defined on the Borel sets of the unit ball $\mathcal{S} = \{x\in E: \|x\|_E=1\}$, such that for any borel set $B$ of $E$:
\begin{equation}\label{eq:radialDec}
\Delta_\alpha(B) = \int_0^\infty \frac{\text{d}r}{r^{1+\alpha}} \int_{\mathcal{S}} \mathbf{1}_B(r y)\ \sigma_\alpha(\text{d} y) \ .
\end{equation}
Besides, $\sigma_\alpha(\text{d} y) = \Delta_\alpha\left(\{x\in E: \|x\|_E\geq 1 \ \text{ and }\ x/\|x\|_E \in \text{d}y\}\right)$.
The previous discussion is summarized in the following proposition.

\begin{proposition}\label{prop:radialDecMpfBm}
Let $(H,E,\mu)$ be any abstract Wiener space such that $E$ is a Hilbert space and the embedding $H\subset E$ is Hilbert-Schmidt. Let $\alpha\in(0,2)$. Then there exists a non-trivial L\'evy measure $\Delta_\alpha$ on $E$ such that Equation (\ref{eq:specRep}) is satisfied, and that can be radially decomposed as in (\ref{eq:radialDec}).
\end{proposition}

In the sequel, $H$ will be specifically the RKHS of the Brownian sheet in $\R^\nu$, that is, $H=L^2([0,1]^\nu)$ endowed with the Lebesgue measure.
We provide $H$ with $E$ and $\mu$ chosen as in Lemma \ref{lem:HS} to get an AWS. For coherence with the definition of the multiparameter fBm, we now use the notations $\Delta\equiv \Delta_{4h},\ h\in(0,1/2)$, for the measures defined in the previous paragraph. We are now ready to prove Theorem \ref{prop:spectralRep}.

\begin{proof}[Proof of Theorem \ref{prop:spectralRep}]
Let us recall that $\mathscr{B}(\xi) = \int_E \left(1-e^{i\langle \xi,x \rangle}\right)\ \textrm{d}\mathbb{B}_x$, where $\mathbb{B}$ is a complex Gaussian white noise on $E$ with control measure $\Delta$. Thus, the variance of the increments of $\mathscr{B}$ reads ( $\overline{(\cdot)}$ denotes complex conjugation):
\begin{align*}
\VV\left(\mathscr{B}_\xi - \mathscr{B}_{\xi'}\right) &= \EE\left((\mathscr{B}_\xi - \mathscr{B}_{\xi'})\overline{(\mathscr{B}_\xi - \mathscr{B}_{\xi'})}\right)\\
&= \int_E \left(e^{i\langle\xi,x\rangle} - e^{i\langle\xi',x\rangle}\right) \left(e^{-i\langle\xi,x\rangle} - e^{-i\langle\xi',x\rangle}\right)\ \Delta(\text{d}x) \\
&= 2 \int_E (1-\cos\langle\xi-\xi',x\rangle)\ \Delta(\text{d}x) = \|S(\xi-\xi')\|_H^{4h} \ .
\end{align*}
Hence this process has the following covariance:
\begin{equation*}
\EE\left(\mathscr{B}_\xi\ \mathscr{B}_{\xi'}\right) = \frac{1}{2}\left(\|S\xi\|_H^{4h} + \|S\xi'\|_H^{4h} - \|S(\xi-\xi')\|_H^{4h}\right) \ .
\end{equation*}

\noindent Now, by density of $E^*$ in $H$, the Paley-Wiener map permits to extend equation (\ref{eq:specRep}) in the following manner:
\begin{equation*}
\forall \varphi\in H,\ \quad \|\varphi\|_H^{4h} = 2\int_E \left(1-\cos\langle\mathcal{I}(\varphi),x \rangle\right)\ \Delta(\text{d}x)
\end{equation*}
Thus, the quantity denoted by $\int_E \left(1-e^{i\langle \mathcal{I}(\varphi),x \rangle}\right)\ \textrm{d}\mathbb{B}_x$ is well-defined and provides an isometric extension of $\mathcal{B}$ to $H$. So for any $f,g\in H$, 
\begin{equation*}
\EE\left(\mathscr{B}(f)\ \mathscr{B}(g)\right) = \frac{1}{2}\left(\|f\|_H^{4h} + \|g\|_H^{4h} - \|f-g\|_H^{4h}\right) \ ,
\end{equation*}
which coincides with the covariance (\ref{eq:cov}) for this choice of $H\ (=L^2([0,1]^\nu)\ )$ and $f=\mathbf{1}_{[0,s]}$ and $g=\mathbf{1}_{[0,t]}$.
\end{proof}

\begin{remark}\label{rk:RKHS}
$\mathcal{W}$ defined by (\ref{eq:L2fBm}) on $H(k)$ and $\mathcal{B}$ on $H$ are different processes: they are not defined on the same spaces, and the first one is a linear application for fixed $\omega$, which is not true for the second. Nevertheless, it appears that $\{\mathcal{W}(k(f,\cdot)),\ f\in L^2([0,1]^\nu)\}$ and $\{\mathcal{B}(f),\ f\in L^2([0,1]^\nu)\}$ are equal in distribution. This implies that they have the same RKHS, that is used in the functional Chung LIL. In particular, this RKHS is given by:
\begin{equation*}
H^\nu = \overline{\text{Span}\left\{ k^{(\nu)}(t,\cdot), \ t\in [0,1]^\nu \right\}} \ ,
\end{equation*}
where $k^{(\nu)}(t,\cdot) = k(\mathbf{1}_{[0,t]}, \mathbf{1}_{[0,\cdot]})$ and the completion is with respect to the scalar product given by:
\begin{equation*}
\left(k^{(\nu)}(t,\cdot), k^{(\nu)}(s,\cdot)\right)_{\nu} = k^{(\nu)}(t,s) \ .
\end{equation*}
\end{remark}

To conclude this section, we present inequalities on $\Delta$ that will be useful in the proof of the LIL. These are extensions of the truncation inequalities of \citet[p.209]{Loeve}. For any $r\in [0,1]$, let us define the subset $A(r)$ of $H$:
\begin{equation}\label{eq:defAr}
A(r) = \left\{ \varphi_{s,t}\ ;\ s,t\in[0,r]^\nu \right\}\ ,
\end{equation}
where $\varphi_{s,t} = \mathbf{1}_{[0,t]\bigtriangleup[0,s]} = |\mathbf{1}_{[0,t]}-\mathbf{1}_{[0,s]}|$. 

\begin{lemma}\label{lem:decoup}
For any $a>0$ and $\varphi\in A(1)$, we have:
\begin{equation}\label{eq:trunc1}
\int_{\|x\|_E < a} \left(1-\cos\langle \mathcal{I}(\varphi),x\rangle\right)\ \Delta(\textrm{d}x) \leq \|\varphi\|_H^{4h}\ \F(a\|\varphi\|_H) \ ,
\end{equation}
where $\F$ is the function defined in Equation (\ref{eq:defF}), and $\F$ continuously decreases to $0$. Besides, there is a constant $C>0$ such that for any $b>0$ and $\varphi\in H$,
\begin{equation}\label{eq:trunc2}
\int_{\|x\|_E > b} \left(1-\cos\langle\mathcal{I}(\varphi),x\rangle\right)\ \Delta(\textrm{d}x) \leq C\ b^{-4h} \ .
\end{equation}
\end{lemma}

\begin{proof}
We start with the first inequality, that we prove by approximation of $\varphi$ by elements of $E^*$. Let $\Phi$ denote the norm of $\varphi$.
Let $(\zeta_n')_{n\in\N} = \{(z_n',\cdot)_E,\ n\in \N\}$ be a sequence of $E^*$ such that $S\zeta_n'$ belongs to the $H$-sphere of radius $\Phi$ and converges to $\varphi$ in $H$. For all $n$, let $\zeta_n$ and $z_n$ be the associated normalized (in $E^*$ and $E$) vectors and $(\lambda'_n)_{n\in\N}$ be the family of norms in $H$: $\lambda_n' = \|S\zeta_n\|_H$. Note that if $\varphi\in H\setminus S(E^*)$, $\lambda_n'\rightarrow 0$ as $n\rightarrow \infty$. By construction, $\lambda_n' = \Phi \|\zeta_n'\|_{E^*}^{-1}>0$.
Then, the radial decomposition (\ref{eq:radialDec}) of $\sigma$ yields:
\begin{align*}
\int_{\|x\|_E < a} (1-\cos\langle\zeta_n',x\rangle)\ \Delta(\text{d}x) &= \int_0^a \frac{\text{d}r}{r^{1+4h}} \int_{\mathcal{S}} \left(1-\cos\left\{\frac{r\Phi}{\lambda_n'} (z_n,y)_E\right\}\right)\ \sigma(\text{d}y) \\
&= \Phi^{4h} \int_0^{a\Phi} \frac{\text{d}u}{u^{1+4h}} \int_{\mathcal{S}} \left(1-\cos\left\{\frac{u}{\lambda_n'} (z_n,y)_E\right\}\right)\ \sigma(\text{d}y) \ ,
\end{align*}
where we applied the change of variable $u = \Phi r$. The last integral converges in $L^2(\sigma)$, so this reads:
\begin{align*}
\int_{\|x\|_E<a} (1-\cos\langle\mathcal{I}(\varphi),x\rangle)\ \Delta(\text{d}x) &= \Phi^{4h} \int_0^{a\Phi} \frac{\text{d}u}{u^{1+4h}} \int_{\mathcal{S}} \left(1-\cos \langle u \mathcal{I}(\varphi/\|\varphi\|_H),y\rangle\right)\ \sigma(\text{d}y) \\
&\leq \Phi^{4h} \F(a\Phi) \ ,
\end{align*}
which gives (\ref{eq:trunc1}). Finally, $\F$ decreases continuously to $0$ since the mapping:
\begin{equation*}
(\varphi,\x)\in A(1)\times [0,1] \mapsto \int_{\|x\|_E<\x} (1-\cos\langle\mathcal{I}(\varphi/\|\varphi\|_H),x\rangle)\ \Delta(\text{d}x)
\end{equation*}
is continuous on a compact (being the image of $[0,1]^\nu\times [0,1]^\nu$ by $\varphi_{\cdot,\cdot}$ , $A(1)$ is compact).
 
\vspace{0.2cm}

\noindent To show (\ref{eq:trunc2}) holds, we use a simple inequality on the cosine function:
\begin{align*}
\int_{\|x\|_E > b} \left(1-\cos\langle\xi,x\rangle\right)\ \Delta(\text{d}x) &\leq 2\ \int_{\|x\|_E > b}  \Delta(\text{d}x)  \\
&\leq 2 \int_b^\infty \frac{\text{d}r}{r^{1+4h}} \sigma(\mathcal{S}) \\
&\leq \frac{2\sigma(\mathcal{S})}{2-4h} b^{-4h} \ .
\end{align*}
This concludes the proof of this lemma.

\end{proof}

\subsection{Small deviations of the multiparameter fBm}

Let us state the following observation on the metric induced by the multiparameter fBm:
\begin{lemma}\label{lem:compDist}
For any $a\in (0,1)$, any $b>a$, there exist $m_{a,b}$ depending on $a$ and $b$ and $M_b$ depending on $b$ only, such that for any $s,t\in [a,b]^\nu$,
\begin{equation*}
m_{a,b} \|s-t\| \leq \lambda\left([0,s]\bigtriangleup [0,t]\right) \leq M_b \|s-t\|
\end{equation*}
In particular, the upper bound holds even if $s,t\in[0,b]$. However, for any given $\alpha\in (0,1)$, we have that for all $\epsilon>0$, there exist $s,t \in [0,1]^\nu$ such that $\lambda\left([0,s]\bigtriangleup [0,t]\right) \leq \varepsilon$ but $\|s-t\| \geq \alpha$.
\end{lemma}

\begin{proof}
The upper and lower bounds on $\lambda\left([0,s]\bigtriangleup [0,t]\right)$ are stated in Lemma 3.1 of \cite{HerbinArras} (up to equivalence of $l^1$ and $l^\infty$ distances with the Euclidean distance), except that there, the constant in the upper bound is said to depend also on $a$. From the proof of \cite{HerbinArras}, it is clear that this is not necessary.

\noindent To prove the last statement, let $s_n=(2^{-n},b,\dots,b) \in [0,b]^\nu$ and $t_n = (b,2^{-n},b,\dots,b)\in [0,b]^\nu$. It appears that $\lambda\left([0,s_n]\bigtriangleup [0,t_n]\right)\rightarrow 0$ as $n\rightarrow \infty$, while $\|s_n-t_n\|$ increases to $\sqrt{2} b$.
\end{proof}

Concerning notations, we will have to compare several distances, so $d_E$ will denote the Euclidean distance in $\R^\nu$, and for any $h\in (0,1]$, $d_{h}$ is the following distance:
\begin{align*}
\textrm{for } s,t\in [0,1]^\nu, \quad d_{h}(s,t) = \lambda\left([0,s]\bigtriangleup [0,t]\right)^{h} \ .
\end{align*}
When $h=1$, we will prefer the notation $d_\lambda$. Note that we will only consider results for $h\leq 1/2$ because of the definition of $\BB$, but $d_h$ is still a distance for $h\in(1/2,1]$ (but no longer negative definite which prevents the definition of a multiparameter fBm for such values). Accordingly, $B_h(t,r)$ is the ball of $d_h$-radius $r$ centred at $t$. If no subscript is written, this will be the Euclidean ball. The notation $\asymp$ between two functions $f$ and $g$ means that near a point $a$, $f(x) = O(g(x))$ and $g(x)= O(f(x))$. We recall that on a (pre-)compact metric space $(T,d)$, the metric entropy $N(T,d,\varepsilon)$ gives, for any $\varepsilon>0$, the minimal number of balls of radius $\varepsilon$ that are necessary to cover $T$.

\begin{lemma}\label{lem:entropy}
Let $\nu\in \N$, then the $d_\lambda$-metric entropy of $[0,1]^\nu$ is, for $\varepsilon$ small enough:
\begin{equation*}
N([0,1]^\nu,d_\lambda,\varepsilon) \asymp \varepsilon^{-\nu} \ .
\end{equation*}
\end{lemma}

\begin{proof}
Let us remark that due to Lemma \ref{lem:compDist}, $d_\lambda(s,t)\leq M_1 d_E(s,t)$, for any $s,t\in [0,1]^\nu$. Thus, for any ball one has $B_\lambda(t_0,r) \supseteq B(t_0, M_1^{-1} r)$. We can assert that:
\begin{equation*}
N\left([0,1]^\nu,d_\lambda,\varepsilon\right) \leq N\left([0,1]^\nu,d_E,M_1^{-1}\varepsilon\right) \asymp \left(M_1^{-1} \varepsilon\right)^{-\nu} \ ,
\end{equation*}
as $\varepsilon\rightarrow 0$. Conversely,
\begin{align*}
N\left([0,1]^\nu,d_\lambda,\varepsilon\right) &\geq N\left([1/2,1]^\nu, d_\lambda,\varepsilon\right) \\
&\geq N\left([1/2,1]^\nu, d_E, m_{1/2,1}^{-1} \varepsilon\right) \\
&\asymp \left(2 m_{1/2,1}^{-1} \varepsilon\right)^{-\nu} \ ,
\end{align*}
so both inequalities give the expected result.
\end{proof}

We now explain how this lemma yields our result on small deviations of $\BB$.

\begin{proof}[Proof of Theorem \ref{th:smallBalls}]
First, notice the isometry between the metric space $([0,1]^\nu,d_h)$ and the subset of $H(k)$ defined by $\{k(\mathbf{1}_{[0,t]},\cdot),\ t\in[0,1]^\nu\}$ with the metric induced by the fBm indexed on $L^2([0,1]^\nu,\lambda)$ (see Eq. (\ref{eq:L2fBm})). Hence, we can apply Theorem 4.6 of \cite{Richard}, which states that (for $h<1/2$), there are positive constants $k_1\leq k_2$ such that for any  $\varepsilon>0$,
\begin{align*}
k_1\ N([0,1]^\nu,d_h,\varepsilon) \leq -\log \PP\left(\sup_{t\in[0,1]^\nu}|\BB_t|\leq \varepsilon\right) \leq k_2\ N([0,1]^\nu,d_h,\varepsilon)\ .
\end{align*}
For any $\varepsilon>0$, any $t\in [0,1]^\nu$, the ball $B_h(t,\varepsilon)$ is the same as $B_\lambda(t,\varepsilon^{1/h})$. A direct consequence is that $N([0,1]^\nu,d_h,\varepsilon) = N([0,1]^\nu,d_\lambda,\varepsilon^{1/h})$. Hence it suffices to calculate the $d_\lambda$-entropy to obtain the result for any $h$. Besides, $\BB$ satisfies the subsequent self-similarity property: for any $r>0$,
\begin{equation*}
\left\{\BB_t, t\in [0,1]^\nu\right\} \overset{(d)}{=} \left\{ r^{-\nu h} \BB_{r t}, t\in [0,1]^\nu\right\} \ .
\end{equation*}
Therefore, $\PP\left(\sup_{t\in[0,r]^\nu}|\BB_t|\leq \varepsilon\right) = \PP\left(\sup_{t\in[0,1]^\nu}|\BB_t|\leq r^{-\nu h} \varepsilon\right)$ and so:
\begin{equation*}
k_1\ N\left([0,1]^\nu,d_\lambda, r^{-\nu}\varepsilon^{1/h}\right)\leq -\log \PP\left(\sup_{t\in[0,r]^\nu}|\BB_t|\leq \varepsilon\right) \leq k_2\ N\left([0,1]^\nu,d_\lambda, r^{-\nu}\varepsilon^{1/h}\right) \ .
\end{equation*}
Lemma \ref{lem:entropy} permits to conclude, with $\kappa_1\leq \kappa_2$ derived from $k_1$, $k_2$ and the approximation on the metric entropy.
\end{proof}

\begin{remark}\label{rem:smallBalls}
\noindent This is different from the L\'evy fBm $X$, for which the above \emph{log}-probability is of the order $r^\nu\varepsilon^{-\nu/h}$ (see \cite{Talagrand95}). In fact, the small deviations of the multiparameter fBm away from the axes are also different of those at $0$, and similar to the Lévy fBm. Indeed, if $t_0$ is not on the axes and $r$ is such that $B(t_0,r)\subset (0,\infty)^\nu$, the equivalence between distances $d_\lambda$ and $d_E$ yields, as $\varepsilon \rightarrow 0$:
\begin{align*}
-\log\PP\left(\sup_{t\in B(t_0,r)} |\BB_t|\leq \varepsilon\right) \asymp N(B(t_0,r), d_E,\varepsilon^{1/h}) \asymp \left(\frac{r}{\varepsilon^{1/h}}\right)^\nu \ .
\end{align*}
\end{remark}

\section{A Chung-type law of the iterated logarithm}\label{sec:ChungLIL}

In this section, we prove Theorem \ref{th:LIL}. The abstract Wiener space is the same as in Theorem \ref{prop:spectralRep} and below.

\begin{remark}
Observe that $h$ is still fixed in $(0,1/2)$. The case $h=1/2$ is special since it corresponds to the Brownian sheet. Its behaviour differs a lot from the $h$-multiparameter fBm with $h<1/2$. This difference is due to the loss of the property of local nondeterminism, which the multiparameter fBm possesses when $h<1/2$ only. For more information on small deviations and Chung-type law of the iterated logarithm of the Brownian sheet, we refer to \cite{Talagrand94}.
\end{remark}

\begin{remark}\label{rem:diffLIL}
If $t_0$ is not on the axes, the Chung-type law of the iterated logarithm given in \cite{HerbinXiao} for the multiparameter fBm reads:
\begin{equation*}
\liminf_{r\rightarrow 0^+}  \frac{\sup_{\|t\|\leq r} |\BB_{t_0+t} - \BB_{t_0}|}{r^h \left(\log\log(r^{-1})\right)^{-h/\nu}} = c\  \quad \text{a.s.} \ ,
\end{equation*}
\noindent for some deterministic $c$ that \emph{does} depend on $t_0$. We will prove below Theorem \ref{th:LIL}, which surprisingly states that near $0$, the local modulus is in fact of order $r^{\nu h}\ \widetilde{\Psi}(r)$ (where $\widetilde{\Psi}$ is a logarithmic correction term), which differs significantly from $r^h$ as soon as $\nu\geq 2$.
\end{remark}

\begin{proof}[Proof of Theorem \ref{th:LIL}]
The proof will be carried out in three steps. In the first, we obtain the lower bound for some constant $\beta_1$. In the second and third steps, we follow the scheme proposed in \cite{Monrad} and which  appears to be standard in subsequent works, but with the addition of methods related to the infinite dimensional setting described above.

\textbf{1)} Let $\gamma>1$, $r_k = \gamma^{-k}$ and $\beta_1 = \left(\kappa_1/(1+\epsilon)\right)^{h/\nu}$, where $\kappa_1$ is the constant in the upper bound of the small deviation probability of $\BB$. The upper bound in the small deviations (\ref{eq:smallBalls}) implies:
\begin{align*}
\sum^{\infty} \PP\left(M(r_k) \leq \beta_1 \psil(r_k)\right) &\leq \sum^{\infty} \exp\left\{- \kappa_1 \beta_1^{-\nu/h} \log \log(r_k^{-1})\right\}\\
&\leq \sum^{\infty} (\log \gamma^k)^{-(1+\epsilon)} <\infty \ ,
\end{align*}
where the sums start at $k$ large enough (i.e. so that $\beta_1 (\log \log \gamma^k)^{-h}$ is small enough, as in Theorem \ref{th:smallBalls}). Then, the Borel-Cantelli lemma gives:
\begin{equation*}
\liminf_{k\rightarrow \infty} M(r_k)/\psil(r_k) \geq \beta_1 \quad \text{a.s.}
\end{equation*}
So for $r_{k+1}<r\leq r_k$:
\begin{equation*}
M(r)/\psil(r) \geq M(r_{k+1})/\psil(r_{k}) \geq \beta_1 \frac{\psil(r_{k+1})}{\psil(r_{k})} \geq \left(\kappa_1/(1+\epsilon)\right)^{h/\nu} \gamma^{-\nu h} \frac{\psilt(r_{k+1})}{\psilt(r_{k})} \ .
\end{equation*}
This is true for any $\epsilon>0, \gamma>1$, hence we get the following lower bound:
\begin{equation}\label{eq:lowBound}
\PP\left(\liminf_{r\rightarrow 0} \frac{M(r)}{\psil(r)} \geq \kappa_1^{h/\nu} \right) = 1\ .
\end{equation}

\textbf{2)} Now, recall that $\kappa_2$ is the constant in the lower bound of the small balls, and define $\beta_2 = \kappa_2^{h/\nu}$. 
For some small (fixed) $\eta>0$, we define the sequence $(\epsilon_k)_{k\in\N^*}$ by:
\begin{equation}\label{eq:defEpsk}
\epsilon_k = \F^{-1}\left((\log k)^{-2h/\nu - 2\eta}\right) \ .
\end{equation}
By Lemma \ref{lem:decoup}, $\F$ is a continuous increasing function on any interval $[0,T]$ such that $\F(0) = 0$. Thus, $\epsilon_k$ is a well-defined sequence which converges to $0$ and satisfies:
\begin{equation*}
\frac{(\log k)^{h/\nu}}{\sqrt{-\F(\epsilon_k) \log \F(\epsilon_k)}} \rightarrow \infty \quad \text{ as } k\rightarrow \infty \ .
\end{equation*}
Let's define another sequence $(r_k)_{k\in\N^*}$ by the following induction:
\begin{equation}\label{eq:defRk}
r_1=1 \quad \text{ and }\quad \forall k\geq 2,\  r_{k+1} = r_k\ \F(\epsilon_k)^{1/(2\nu h)}\ \epsilon_{k+1}^{2/\nu} \ .
\end{equation}
One can now choose $\psiut$ to be any increasing continuous function on $[0,1]$, satisfying the following set of conditions: for any $k\in\N^*$,
\begin{equation}\label{eq:defPsiut}
\psiut(r_k) = \left(\log k\right)^{-h/\nu} \ .
\end{equation}
We recall that for a given $\psiut$, chosen as above, $\psiu$ is defined by $\psiu(r) = r^{\nu h}\ \psiut(r)\ ,\ {r\in[0,1]}$.

For these parameters, the lower bound in the small deviations of $\BB$ implies:
\begin{align}\label{eq:divergence}
\sum^\infty \PP\left(\sup_{t\in[0,r_k]^\nu} |\BB_t|/\psiu(r_k) \leq \beta_2\right) &\geq \sum^\infty \exp\left\{-\kappa_2 (\beta_2 \psiut(r_k))^{-\nu/h} \right\} \nonumber\\
&\geq \sum^\infty \frac{1}{k} =\infty \ ,
\end{align}
where the sums start at $k$ large enough (i.e. so the hypothesis of Theorem \ref{th:smallBalls} is satisfied). This is not enough to prove the expected result, because these events are not independent. We will fix this using an idea that appeared in \cite{Talagrand95} and \cite{Monrad}, to create independence by means of increment stationarity. Since this last property is not satisfied by the multiparameter fractional Brownian motion, we shall rely instead on the spectral representation obtained in the previous section.

\noindent We recall that $\varphi_t = \mathbf{1}_{[0,t]},\ t\in[0,1]^\nu$ is an element of $H$. For a family of disjoint intervals $\{I_k = (a_k,a_{k+1}],\ k\in \N\}$, where $(a_k)_{k\in \N}$ is an increasing sequence of $\R_+$ such that $a_k\rightarrow \infty$ ($a_k$ will be specified later), we define the following processes:
\begin{align}
& \BB^{k}_{t} = \int_{\|x\|_E \in I_k} \left(1-e^{i\langle\mathcal{I}(\varphi_t), x\rangle}\right)\ \text{d}\mathbb{B}_x \ ,\ t\in [0,1]^\nu \label{eq:defBhk}\\
& \tilde{\BB}^{k}_{t} = \BB_{t} - \BB^{k}_t = \int_{\|x\|_E \notin I_k} \left(1-e^{i\langle\mathcal{I}(\varphi_t), x\rangle}\right)\ \text{d}\mathbb{B}_x \ ,\ t\in[0,1]^\nu\ .\label{eq:defBhktilde}
\end{align}

\noindent Let $\Sigma$ denote the covariance operator of $\BB$ and $\Sigma_k$ denote the covariance operator of $\BB^{k}$. It is clear that $\Sigma-\Sigma_k$ is a positive semi-definite operator. Hence, Anderson's correlation inequality \cite{Anderson} applies and we get, for all $k\in \N$:
\begin{equation*}
\PP\left(\sup_{t\in[0,r_k]^\nu} |\BB^{k}_{t}|/\psiu(r_k) \leq \beta_2\right) \geq \PP\left(\sup_{t\in[0,r_k]^\nu} |\BB_t|/\psiu(r_k) \leq \beta_2 \right)\ .
\end{equation*}

\noindent As a consequence of Equation (\ref{eq:divergence}), we see that:
\begin{equation*}
\sum_{k\geq 1} \PP\left(\sup_{t\in[0,r_k]^\nu} |\BB^{k}_{t}|/\psiu(r_k) \leq \beta_2\right) = \infty\ .
\end{equation*}
Since the events $\left\{\sup_{t\in[0,r_k]^\nu} |\BB^{k}_{t}|/\psiu(r_k)\leq \beta_2\right\}, \ k\in\N,$ are independent, the reciprocal of Borel-Cantelli lemma yields that almost surely,
\begin{equation}\label{eq:borneSupPartielle}
\liminf_{k\rightarrow \infty} \sup_{t\in[0,r_k]^\nu} |\BB^{k}_t|/\psiu(r_k) \leq \beta_2 \ .
\end{equation}

\vspace{0.1cm}

\textbf{3)} For any $k\in\N^*$, let $a_{k} = r_k^{-\nu/2} \epsilon_k$. Note that (\ref{eq:defRk}) implies that:
\begin{equation}\label{eq:divergenceAk}
a_{k+1} r_k^{\nu/2} = r_{k+1}\ \F(\epsilon_k)^{-1/4h}\ \epsilon_{k+1}^{-1} \geq \F(\epsilon_k)^{-1/4h} \ .
\end{equation}
In particular, $a_{k+1} r_k^{\nu/2}$ goes to infinity. Now, Lemma \ref{lem:decoup} acts on the incremental variance of $\tilde{\BB}^{k}$ as follows: for any $s,t\in [0,r_k]^\nu$, letting $\varphi_{s,t} = |\varphi_s - \varphi_t|$,
\begin{align}\label{eq:partialInc}
\VV\left(\tilde{\BB}^{k}_s - \tilde{\BB}^{k}_t\right) &= \int_{\|x\|_E<a_k} \left(1-\cos\langle \mathcal{I}(\varphi_{s,t}),x\rangle\right)\ \Delta(\text{d}x) + \int_{\|x\|_E\geq a_{k+1}} \left(1-\cos\langle \mathcal{I}(\varphi_{s,t}),x\rangle\right)\ \Delta(\text{d}x) \nonumber\\
&\leq C\left(\|\varphi_{s,t}\|_H^{4h}\ \F\left(a_k\ \|\varphi_{s,t}\|_H\right) + a_{k+1}^{-4h}\right) \\
&\leq C\left(r_k^{2\nu h}\ \F(a_k r_k^{\nu/2}) + a_{k+1}^{-4h}\right) \nonumber\\
&\leq C r_k^{2\nu h} \left( \F(\epsilon_k) + (a_{k+1}\ r_k^{\nu/2} )^{-4h}\right) \ ,\nonumber
\end{align}
for some positive constant $C$, where $\|\varphi_{s,t}\|_H^2 = \lambda([0,s]\bigtriangleup [0,t]) \leq \lambda([0,r_k]^\nu) = r_k^\nu$. Thus, for this choice of $r_k$ and $a_k$, letting $D_k^2$ denote this incremental variance, the previous equation and (\ref{eq:divergenceAk}) give:
\begin{align*}
D_{k}^2 &= \sup_{s,t\in [0,r_k]^\nu} \VV\left(\tilde{\BB}^{k}_s - \tilde{\BB}^{k}_t\right) \\
&\leq 2 C\ r_k^{2\nu h}\ \F(\epsilon_k) \ ,
\end{align*}
which decreases faster than $\sup_{s,t\in [0,r_k]^\nu} \VV\left(\BB_s - \BB_t\right)$ (as $k\rightarrow \infty$). By a Gaussian concentration result, we will see that $D_k$ will permit us to obtain an upper bound for the large deviations of $\tilde{\BB}^{k}$. Let $\tilde{d}_{h,k}$ be the distance induced by this process. We have just seen that $\tilde{d}_{h,k} \leq d_h$. Thus, the metric entropy of a set computed with $\tilde{d}_{h,k}$ is smaller than the one computed with $d_h$.
\begin{align*}
\int_0^{D_{k}} \sqrt{\log N([0,r_k]^\nu,\tilde{d}_{h,k},\varepsilon)}\ \text{d}\varepsilon &\leq \int_0^{D_{k}} \sqrt{\log N([0,r_k]^\nu,d_h,\varepsilon)}\ \text{d}\varepsilon \\
&\leq \int_0^{D_{k}} \sqrt{\log \left(\kappa \frac{r_k^{\nu^2}}{\varepsilon^{\nu/h}}\right)}\ \text{d}\varepsilon \ ,
\end{align*}
where the upper bound on $N([0,r_k]^\nu,d_{h},\varepsilon)$ is due to the link with the small balls of $\BB$ as in Theorem \ref{th:smallBalls}, with some $\kappa>0$ that comes from the asymptotics of $N([0,1]^\nu,d_\lambda,\varepsilon)$ in Lemma \ref{lem:entropy}. 
\begin{align*}
\int_0^{D_{k}} \sqrt{\log N([0,r_k]^\nu,\tilde{d}_{h,k},\varepsilon)}\ \text{d}\varepsilon &\leq \sqrt{\frac{\nu}{h}} \kappa^{h/\nu} r_k^{\nu h} \int_0^{\sqrt{2C} \kappa^{-h/\nu} \sqrt{\F(\epsilon_k)}} \sqrt{\log x^{-1}}\ \text{d}x \\
&\leq C_1 r_k^{\nu h} \sqrt{-\F(\epsilon_k)\ \log(C_2 \F(\epsilon_k)) } \ ,
\end{align*}
where we made the change of variable $x= \varepsilon\ \kappa^{-h/\nu} r_k^{-\nu h}$, and $C_1$ and $C_2$ are given by:
\begin{equation*}
C_1 = \sqrt{\frac{C\nu}{h}}\ , \quad \quad C_2 = 2 C \kappa^{-2h/\nu} \ .
\end{equation*}

\vspace{0.2cm}

\noindent By Talagrand's lemma \cite{Talagrand95}, if $u> u_0(k) := C_1 r_k^{\nu h} \sqrt{-\F(\epsilon_k)\ \log(C_2 \F(\epsilon_k)) }$, 
\begin{align*}
\PP\left(\sup_{t\in[0,r_k]^\nu}|\tilde{\BB}^{k}_t|\geq u\right) \leq \exp\left(-\frac{(u-u_0(k))^2}{D_{k}^2}\right) \ .
\end{align*}
Let $\epsilon>0$. In order to replace $u$ by $\epsilon \beta_2 \psiu(r_k)$, one notices that:
\begin{equation*}
\frac{\psiu(r_k)}{u_0(k)} = \frac{\psiut(r_k)}{C_1 \sqrt{-\F(\epsilon_k)\ \log(C_2 \F(\epsilon_k)) }} \ , 
\end{equation*}
and this quantity goes to infinity, by definition of $\epsilon_k$ in (\ref{eq:defEpsk}) and $\psiut$ in (\ref{eq:defPsiut}). Thus, replacing $u$ with $\epsilon\beta_2 \psiu(r_k)$ for $k$ big enough, reads:
\begin{align*}
\PP\left(\sup_{t\in[0,r_k]^\nu}|\tilde{\BB}^{k}_t|\geq \epsilon \beta_2 \psiu(r_k)\right) &\leq \exp\left(-\frac{(\epsilon \beta_2 \psiu(r_k)-u_0(k))^2}{D_{k}^2}\right) \\
&\leq \exp\left( - \frac{u_0(k)^2}{D_k^2} \ \left(\epsilon \beta_2\psiu(r_k)\ u_0(k)^{-1}-1\right)^2\right) \\
&\leq \exp\left( - \frac{C_1\sqrt{-\log(C_2 \F(\epsilon_k))}}{2 C} \ \left(\epsilon \beta_2\psiu(r_k)\ u_0(k)^{-1}-1\right)^2\right) \ ,
\end{align*}
whose sum is finite, since $\psiu(r_k) u_0(k)^{-1}$ diverges and $\F(\epsilon_k)$ goes to $0$. Hence, applying once again the Borel-Cantelli lemma, we have almost surely,
\begin{equation*}
\liminf_{k\rightarrow \infty} \sup_{t\in [0,r_k]^\nu} |\tilde{\BB}^{k}_t|/\psiu(r_k) \leq \epsilon \beta_2 \ .
\end{equation*}
Therefore, combining this with (\ref{eq:borneSupPartielle}), we see that almost surely:
\begin{equation*}
\liminf_{k\rightarrow \infty} \sup_{t\in [0,r_k]^\nu} |\BB_t|/\psiu(r_k) \leq (1+\epsilon)\beta_2 \ .
\end{equation*}

\noindent Since this is true for any $\epsilon>0$, we obtain the expected upper bound:
\begin{equation}\label{eq:borneSupPsiu}
\PP\left(\liminf_{r\rightarrow 0} \frac{M(r)}{\psiu(r)} \leq \kappa_2^{h/\nu} \right) = 1 \ .
\end{equation}
\end{proof}

\begin{remark}

Let $(r_k)_{k\in\N}$ be defined as in (\ref{eq:defRk}) and $\beta_1 = \left(\kappa_1/(1+\epsilon)\right)^{h/\nu}$.
Using also the relation (\ref{eq:defPsiut}) and proceeding as in \emph{\textbf{1)}}, we obtain:
\begin{equation*}
\PP\left(\liminf_{k\rightarrow \infty} \frac{M(r_k)}{\psiu(r_k)} \geq \kappa_1^{h/\nu} \right) = 1 \ ,
\end{equation*}
but note that this is not sufficient to conclude that $\psiu$ is the good modulus.
\end{remark}

We end this section with a discussion on the consequences of the rate of decay of $\F$. 
To make this \emph{lim inf} result precise, one would need to find $\psiut$ explicitly, which depends only on the rate of decay of $\F$ near $0$. For instance, if we were able to prove that $\F(\x) \leq \x^\gamma$ for some $\gamma>0$, as $\x\rightarrow 0$, then it would be possible to show that $\psiut(r) = \left(\log\log (r^C)\right)^{-h/\nu}$, where $C = -\nu^{-2}(1+4h/\gamma)$, is a function for which (\ref{eq:defPsiut}) holds. Since in that case $\psiut(r) \sim \psilt(r)$ as $r\rightarrow 0$, we would get 
\begin{equation*}
\PP\left(\liminf_{r\rightarrow 0} \frac{M(r)}{r^{\nu h} (\log\log(r^{-1}))^{-h/\nu}} \in [\kappa_1^{h/\nu}, \kappa_2^{h/\nu}] \right) = 1 \ .
\end{equation*}
Note that in this situation, a $0-1$ law (which is explained in Remark \ref{rem:01} below) implies that the above limit is constant almost surely.
A faster rate would yield the same conclusion, while a slower rate for $\F$ would certainly mean that $\psilt$ converges to $0$ too quickly.

\begin{remark}{\emph{($0-1$ law of the multiparameter fBm.)}}\label{rem:01}
If one had $\F(\x)\leq \x^\gamma$, a $0-1$ law very similar to the one presented in \cite{LuanXiao2010} would hold. Indeed, letting $\mathcal{F}_n$ be the $\sigma$-algebra generated by the process $\sum_{k=n}^\infty \BB^{k}$ and $\mathcal{F}^\infty = \cap_{n\geq 1} \mathcal{F}_n$ be the tail $\sigma$-algebra, it follows from Kolmogorov's $0-1$ law that any event $A$ in $\mathcal{F}^\infty$ is trivial. Thus, if the event:
\begin{equation*}
A = \left\{\liminf_{r\rightarrow 0} M(r)/\Psi(r) = \text{constant}\right\}
\end{equation*}
belonged to $\mathcal{F}^\infty$, this would mean that there is an exact modulus in the Chung-type law. If $\F(\x)\leq \x^\gamma$, Lemma  \ref{lem:decoup}, the first part of Equation (\ref{eq:partialInc}) and Kolmogorov's continuity criterion yield that $A\in \mathcal{F}^\infty$.

\end{remark}

\section{Functional law of the iterated logarithm}\label{sec:FLIL}

We prove Theorem \ref{th:functLIL}. This proof follows closely \cite{Monrad}, with the necessary adaptations similar to the ones of the previous part. Yet we include it for completeness.\\
The following technical lemma is adapted from \cite{deAcosta83,Monrad}. The norm of $H^\nu$ (see Remark \ref{rk:RKHS}) is denoted by $\|\cdot\|_{\nu}$ and we will also abbreviate $\sup_{t\in[0,1]^\nu}|f(t)| = \|f\|_\infty$.\\
As in the previous part, we also have values for $\gamma^{(\ell)}(\varphi)$ and $\gamma^{(u)}(\varphi)$:
\begin{equation*}
\gamma^{(\ell)}(\varphi) = \frac{1}{\sqrt{2}}\kappa_1^{h/\nu} (1-\|\varphi\|_{\nu}^2)^{-h/\nu} \ \text{ and }\ \gamma^{(u)}(\varphi) = \frac{1}{\sqrt{2}}\kappa_2^{h/\nu} (1-\|\varphi\|_{\nu}^2)^{-h/\nu} \ .
\end{equation*}

\begin{lemma}\label{lem:techFLIL}
For $0<s<r<u<e^{-1}$ and $\varphi\in H^\nu$,
\begin{align*}
\left(\log \log r^{-1}\right)^{h/\nu +1/2} \|\etal_r - \varphi\|_\infty &\geq \left(\frac{s \log\log u^{-1}}{u\log\log s^{-1}}\right)^{h\nu} \left(\log\log s^{-1}\right)^{h/\nu+1/2} \|\etal_s - \varphi\|_\infty\\
&\quad - M_1 \left(\log\log u^{-1}\right)^{h/\nu+1/2} \left(\frac{u-s}{u}\right)^{h} \|\varphi\|_{\nu}\\
&\quad - \left(\log\log u^{-1}\right)^{h/\nu+1/2} \sqrt{\left(1-\left(\frac{s}{u}\right)^{2\nu h} \frac{\log\log s^{-1}}{\log\log u^{-1}} \right)} \|\varphi\|_\infty\ ,
\end{align*}
where $M_1$ is the constant in Lemma \ref{lem:compDist} which corresponds to $b=1$.
\end{lemma}

\noindent For the proof of this lemma, one can refer to appendix \ref{app:proof541} .

We recall the following nice proposition from \cite{Monrad}, concerning the Gaussian measure of shifted convex sets\footnote{It existed before in the literature, in a more general form. See the references therein.}:
\begin{proposition}\label{prop:shifted}
Let $\mu$ be a Gaussian measure on a separable Banach space $E$. For any convex, symmetric, bounded and measurable subset $V$ of $E$ of positive measure, if $\varphi$ belongs to the RKHS of $\mu$, then
\begin{equation*}
\lim_{t\rightarrow \infty} t^{-2} \left( \log \mu(V + t\varphi) - \log \mu(V) \right) = -\frac{1}{2}\|\varphi\|_\mu^2 \ .
\end{equation*}
\end{proposition}

\begin{proof}[Proof of Theorem \ref{th:functLIL}]
This proof is divided into two parts: the first one to give the lower bound on $\gamma(\varphi)$, and the second one for the upper bound. 

\vspace{0.2cm}

\hspace{0.5cm} \textbf{I) Proof of the lower bound}

Let $\epsilon >0$ and $\gamma_1$ defined by:
\begin{equation*}
\gamma_1 = \left(\frac{\kappa_1}{(1+\epsilon)}\right)^{h/\nu} (1-\|\varphi\|_{\nu}^2)^{-h/\nu} \ .
\end{equation*}
Recall that $\psilt(r) = (\log\log r^{-1})^{-h/\nu}$, so that the following events, defined for $k\in \N$ by:
\begin{equation*}
A_k = \left\{ \psilt(r_k)^{-1-\nu/2h} \|\etal_{r_k} - \varphi\|_\infty \leq \gamma_1 \right\}
\end{equation*}
for some decreasing sequence $r_k$ (explicited later), will be written:
\begin{equation*}
A_k = \left\{ \left\|r_k^{-\nu h} \BB(r_k\cdot) - \sqrt{2 \log\log r_k^{-1}} \varphi\right\|_\infty \leq \gamma_1 (\log\log r_k^{-1})^{-h/\nu} \right\} \ .
\end{equation*}

\noindent Let $\delta>0$ and $\delta< \epsilon(1-\|\varphi\|_{\nu}^2)$. By Proposition \ref{prop:shifted}, and then by the small deviations of $\BB$, we have for $k$ large enough (depending on $\delta$),
\begin{align*}
\log \PP(A_k) &\leq \log \PP\left(\sup_{t\in [0,1]^\nu} |\BB(r_k t)|\leq \gamma_1 r_k^{h\nu} (\log\log r_k^{-1})^{-h/\nu} \right) - (\log\log r_k^{-1} ) (\|\varphi\|_{\nu}^2 - \delta) \\
&\leq -(1+\epsilon)(1-\|\varphi\|_{\nu}^2) (\log\log r_k^{-1})- (\log\log r_k^{-1} ) (\|\varphi\|_{\nu}^2 - \delta) \ .
\end{align*}
This implies that
\begin{equation*}
\PP(A_k) \leq \exp\left\{ -\left(1+\epsilon(1 -\|\varphi\|_{\nu}^2) -\delta\right) \log\log r_k^{-1} \right) \ .
\end{equation*}
Now put:
\begin{equation*}
r_k = \exp\left\{ -k y(k)\right\} \ ,
\end{equation*}
where
\begin{equation*}
y(k) = \frac{\log\log k}{(\log k)^{h^{-1} +1}} \ .
\end{equation*}
Since $\delta$ was chosen appropriately, $\epsilon(1 -\|\varphi\|_{\nu}^2) -\delta$ is positive, and 
\begin{equation*}
\sum^\infty \PP(A_k) <\infty \ ,
\end{equation*}
where the sum is over $k$ large enough, according to the previous remarks. Therefore, almost surely,
\begin{equation*}
\liminf_{k\rightarrow \infty} \left(\log\log r_k^{-1}\right)^{h/\nu+1/2} \sup_{t\in [0,1]^\nu} |\eta_{r_k}(t) -\varphi(t)| \geq \frac{1}{\sqrt{2}}\gamma_1 \ .
\end{equation*}
To obtain the result for $r\rightarrow 0$, we use Lemma \ref{lem:techFLIL} with $u=r_k$, $s=r_{k+1}$ and $r$ in between. Then
\begin{align}
\left(\log \log r^{-1}\right)^{h/\nu +1/2} \|\eta_r - \varphi\|_\infty &\geq \left(\frac{r_{k+1} \log\log r_k^{-1}}{r_k\log\log r_{k+1}^{-1}}\right)^{h\nu} \left(\log\log r_{k+1}^{-1}\right)^{h/\nu+1/2} \|\eta_{r_{k+1}} - \varphi\|_\infty \tag{$\star$}\label{eq:techFLIL1}\\
&\quad - M_1 \left(\log\log r_k^{-1}\right)^{h/\nu+1/2} \left(\frac{r_k-r_{k+1}}{r_k}\right)^{h} \|\varphi\|_{\nu} \tag{$\star\star$}\label{eq:techFLIL2}\\
&\quad - \left(\log\log r_k^{-1}\right)^{h/\nu+1/2} \sqrt{\left(1-\left(\frac{r_{k+1}}{r_k}\right)^{2\nu h} \frac{\log\log r_{k+1}^{-1}}{\log\log r_k^{-1}} \right)} \|\varphi\|_\infty\ . \tag{$\star\star\star$} \label{eq:techFLIL3}
\end{align}
Note that by the inequality $e^{-x}\geq 1-x$, and the decrease of $y(k)$ (for $k$ large),
\begin{align*}
\frac{r_{k+1}}{r_k} &\geq 1-\left\{y(k+1)\log\left(y(k+1)\right) - y(k)\log\left(y(k)\right)\right\}\\
&\geq 1-y(k+1) \ .
\end{align*}
Thus, the ratio in Equation (\ref{eq:techFLIL1}) converges to $1$. Likewise, the ratio in (\ref{eq:techFLIL2}) is smaller than $y(k+1)^h$, so that:
\begin{align*}
\left(\frac{r_k-r_{k+1}}{r_k}\right)^{h} \left(\log\log r_k^{-1}\right)^{h/\nu+1/2}  &\leq y(k+1)^h \left(\log \left(k y(k)\right)\right)^{h/\nu+1/2}\\
&\leq \frac{\left(\log\log(k+1)\right)^h}{\left(\log(k+1)\right)^{h+1}} \ (\log k)^{h/\nu+1/2} \left(1 + \frac{\log y(k)}{\log k}\right)^{h/\nu+1/2} \ ,
\end{align*}
which clearly goes to $0$. For the last term (\ref{eq:techFLIL3}),
\begin{align*}
\left(\log\log r_k^{-1}\right)^{2h/\nu+1} &\left(1-\left(\frac{r_{k+1}}{r_k}\right)^{2\nu h} \frac{\log\log r_{k+1}^{-1}}{\log\log r_k^{-1}} \right) \\
&\leq \left(\log\log r_k^{-1}\right)^{2h/\nu} \bigg\{ \log(k y(k)) - \log\left((k+1)y(k+1)\right) \\
&\quad\quad\quad\quad\quad\quad\quad\quad\quad\quad\quad\quad+ \log\left((k+1)y(k+1)\right) \left[1-\left(1-y(k+1)\right)^{2h\nu}\right]  \bigg\}  \\
&\leq \left(\log\log r_k^{-1}\right)^{2h/\nu} \bigg\{ \log(k y(k)) - \log\left((k+1)y(k+1)\right) \\
&\quad\quad\quad\quad\quad\quad\quad\quad\quad\quad\quad\quad+ 2\nu h\ y(k+1)\ \log\left((k+1)y(k+1)\right) \bigg\} \ .
\end{align*}
One can show that $\log\left(k y(k)\right) - \log\left((k+1) y(k+1)\right) \sim -k^{-1}$, thus
\begin{equation*}
\left(\log\log r_k^{-1}\right)^{2h/\nu} \left\{ \log(k y(k)) - \log\left((k+1)y(k+1)\right) \right\}
\end{equation*}
converges to $0$, and so does the remaining term, since:
\begin{align*}
\left(\log\log r_k^{-1}\right)^{2h/\nu}\ y(k+1)\ \log\left((k+1)y(k+1)\right) \sim \left(\log k\right)^{2h/\nu-1-h^{-1}+1}\  \log\log(k+1) 
\end{align*}
and the sum of the exponents $1+2h/\nu-h^{-1}-1$ is strictly negative ($\nu\geq 1$ and $h<1/2$).

\vspace{0.2cm}

\hspace{0.5cm} \textbf{II) Proof of the upper bound}
\ 
\vspace{0.2cm}

The proof of Theorem \ref{th:LIL} and Proposition \ref{prop:shifted} allow to make a quick proof for this bound.
Let us define $\gamma_2 = \kappa_2^{h/\nu} \left(1-\|\varphi\|_{\nu}^2\right)^{-h/\nu}$, and put $r_k$ and $a_k$ as in steps $2)$ and $3)$ of the proof of the LIL. Again, let $\BB^{k}$ and $\tilde{\BB}^{k}$ be the processes defined by (\ref{eq:defBhk}) and (\ref{eq:defBhktilde}). As in \cite{Monrad}, we define the following events, for any $\epsilon>0$:
\begin{align*}
&A_k(\epsilon) = \left\{\Big\|r_k^{-\nu h} \BB(r_k\cdot) - \sqrt{2} \left(\psiut(r_k)\right)^{-\nu/2h}\varphi\Big\|_\infty \leq \gamma_2 (1+\epsilon)\ \psiut(r_k)\right\} \\
&B_k(\epsilon) = \left\{\Big\|r_k^{-\nu h} \BB^{k}(r_k\cdot) - \sqrt{2} \left(\psiut(r_k)\right)^{-\nu/2h}\varphi\Big\|_\infty \leq \gamma_2 (1+\epsilon)\ \psiut(r_k)\right\} \\
&C_k(\epsilon) = \left\{\left\|r_k^{-\nu h} \tilde{\BB}^{k}(r_k\cdot)\right\|_\infty \geq \gamma_2 \epsilon\ \psiut(r_k)\right\} \ .
\end{align*}

\noindent This time, apply Theorem \ref{th:smallBalls} and Proposition \ref{prop:shifted} to deduce the existence of a small $\delta>0$ such that for $k$ large enough, the following lower bound on the probability of the event $A_k(\epsilon)$ holds:
\begin{align*}
\log\PP(A_k(\epsilon)) &\geq \log\PP\left(\sup_{t\in[0,1]^\nu}|\BB(r_k t)|\leq \gamma_2(1+\epsilon) r_k^{h\nu} \psiut(r_k)\right) - \left(\psiut(r_k)\right)^{-\nu/h}\ \left(\|\varphi\|_{\nu}^2+\delta\right) \\
&\geq -(1+\epsilon)^{-\nu/h} \left(1-\|\varphi\|_{\nu}^2\right) \left(\psiut(r_k)\right)^{-\nu/h}- \left(\psiut(r_k)\right)^{-\nu/h}\ \left(\|\varphi\|_{\nu}^2+\delta\right) \\
&\geq -\log k\ \left((1+\epsilon)^{-\nu/h} \left(1-\|\varphi\|_{\nu}^2\right) - \left(\|\varphi\|_{\nu}^2+\delta\right)\right) \ .
\end{align*}
Therefore, choosing $\delta$ small enough to ensure that $-(1+\epsilon)^{-\nu/h} \left(1-\|\varphi\|_{\nu}^2\right)-  \left(\|\varphi\|_{\nu}^2+\delta\right)$ is greater than $-1$ implies that:
\begin{align*}
\sum_{k=1}^\infty \PP(A_k(\epsilon)) &\geq \sum_{k=1}^\infty k^{-(1+\epsilon)^{-\nu/h} \left(1-\|\varphi\|_{\nu}^2\right)-  \left(\|\varphi\|_{\nu}^2+\delta\right)} \\
&=\infty \ .
\end{align*}

\noindent All that remains to notice is that:
\begin{align*}
A_k(\epsilon)\subset B_k(2\epsilon)\ \cup\ C_k(\epsilon) \subset A_k(3\epsilon)\ \cup\ C_k(\epsilon) \ ,
\end{align*}
and that the choice of $a_k$ and $r_k$ implies that $\sum \PP(C_k(\epsilon))<\infty$ (as in the proof of Theorem \ref{th:LIL}). The rest follows strictly the proof of \cite{Monrad}.
\end{proof}

As in Remark \ref{rem:01}, if $\F$ were proven to have fast decay, the same $0-1$ law that we used for the Chung law would give the same conclusion, i.e. that there is a constant between $\gamma^{(\ell)}(\varphi)$ and $\gamma^{(u)}(\varphi)$ such that almost surely:
\begin{equation*}
\liminf_{r\rightarrow 0^+} \ (\log\log(r^{-1}))^{h/\nu+1/2} \sup_{t\in[0,1]^\nu}|\etal_r(t) - \varphi(t)| = \gamma(\varphi) \ .
\end{equation*}

We end this part on laws of the iterated logarithm with a remark concerning the previous result when $\|\varphi\|_{\nu} = 1$. This case was studied a lot in the literature, as it yields a different rate of convergence. In fact, for $\|\varphi\|_{\nu} = 1$, part {\bfseries I)} of the previous proof can be directly adapted to give:
\begin{equation*}
\liminf_{r\rightarrow 0^+} \ (\log\log(r^{-1}))^{h/\nu+1/2} \sup_{t\in[0,1]^\nu}|\eta_r(t) - \varphi(t)| = \infty \quad \text{a.s.} 
\end{equation*}
The exact rate was computed in many situations and it is likely that standard techniques (as in \cite{deAcosta83,Monrad}) and the present spectral representation and small deviations will permit to compute the exact rate in the functional law of the iterated logarithm on the unit sphere for the multiparameter fBm.

\section*{Conclusive remarks}

The method to prove a Chung-type LIL often relies on the same estimates as the method to compute the exact Hausdorff measure of the range of a Gaussian process with stationary increments (see \cite[Prop. 3.1]{Xiao96} and \cite[Prop. 4.1]{Talagrand95}, and \cite{Falconer} on Hausdorff measures and Hausdorff dimension).\\
Let $\BB^{(d)}$ be a $d$-dimensional multiparameter fBm. If $\BB^{(d)}$ behaved away from the axes as it behaves at $0$, then the Hausdorff dimension of $\BB^{(d)}([0,1]^\nu)$ would be equal to $h^{-1}\wedge d$. 
However, this is not the case by the $\sigma$-stability of the Hausdorff dimension, since it is easy to find a set $A\subset [0,1]^\nu$ such that $\dH\left(\BB^{(d)}(A)\right) = \frac{\nu}{h}\wedge d\ a.s.$, which is strictly larger than $h^{-1}$ when $h<d$ and $\nu>1$ (see \cite[pp. 134-136]{RichardThesis}). Thus the behaviour away from the axes prevails and $\dH\left(\BB^{(d)}([0,1]^\nu)\right) = \frac{\nu}{h}\wedge d\ \ a.s.$

Another aspect of the singular behaviour of $\BB$ at $0$ is linked to its sample path regularity. Indeed, the two standard methods for measuring the exact Hölder continuity at one point yield different results. Namely, let $B_+(t_0,\rho)$ denote the intersection of $[0,1]^\nu$ with the Euclidean ball centred in $t_0$ with radius $\rho$, and let $\nu h<1$. Then, the \emph{pointwise} Hölder exponent at $t_0\in [0,1]^\nu$, denoted by $\alphar_{\BB}(t_0)$, and the \emph{local} Hölder exponent $\widetilde{\alphar}_{\BB}(t_0)$, satisfy almost surely:
\begin{align*}
&\alphar_{\BB}(t_0):=\sup\left\{ \alpha>0:\;\limsup_{\rho\rightarrow 0}\sup_{s,t\in B_+(t_0,\rho)} \frac{|\BB(t)-\BB(s)|}{\rho^{\alpha}}<\infty \right\} = \left\{
\begin{tabular}{ll}
$h$ & if $t_0\neq 0$\\
$\nu h$ & if $t_0=0$
\end{tabular}
\right. \\
&\widetilde{\alphar}_{\BB}(t_0):=\sup\left\{ \alpha>0:\;\limsup_{\rho\rightarrow 0}\sup_{s,t\in B_+(t_0,\rho)} \frac{|\BB(t)-\BB(s)|}{|t-s|^{\alpha}}<\infty \right\} 
 = h\ .
\end{align*}

\noindent Hence, the local and pointwise random Hölder exponents (we refer to \cite{ehjlv} for complementary information on these exponents, and \cite{RichardThesis} for a proof of the above result) differ almost surely at $0$. This behaviour is typical of functions which have oscillations and variations at different scales, such as the chirp function $f:x\mapsto |x|^\alpha \sin\left(|x|^{-\beta}\right)$, where $\alpha\in(0,1),\ \beta>0$.
It is possible that the pointwise Hölder exponent captures the local oscillations in Chung-type LILs, while the local exponent cannot. This could provide a Chung-type law for a wide class of Gaussian processes having the local nondeterminism property and a spectral representation similar to the multiparameter fBm.

\paragraph{Acknowledgements.} The author is grateful to the anonymous referees who helped him improve the quality and organization of this paper.

\begin{appendices}

\titleformat{\section}[hang]{\large \normalfont \scshape}{Appendix \thesection}{10pt}{}

\section{Proof of Lemma \ref{lem:techFLIL}}\label{app:proof541}

For the original proof, see Lemma 5.3 of \cite{deAcosta83}. We make here the necessary modifications.

\begin{align*}
\left(\log\log r^{-1}\right)^{h/\nu+1/2} \|\etal_r - f\|_\infty &= \frac{(\log\log r^{-1})^{h/\nu}}{r^{\nu h}} \left\|\BB(r\cdot) - r^{\nu h}\sqrt{\log \log r^{-1}} f\right\|_\infty \\
&\geq \frac{(\log\log r^{-1})^{h/\nu}}{r^{\nu h}} \left\|\BB(s\cdot) - r^{\nu h}\sqrt{\log \log r^{-1}} f\left(\frac{s}{r}\cdot\right)\right\|_\infty \\
& \geq \frac{(\log\log u^{-1})^{h/\nu}}{u^{\nu h}} \left\|\BB(s\cdot) - r^{\nu h}\sqrt{\log \log r^{-1}} f\left(\frac{s}{r}\cdot\right)\right\|_\infty \ .
\end{align*}
Now choosing $a= s^{\nu h}\sqrt{\log\log s^{-1}}$ and $b=u^{\nu h}\sqrt{\log\log u^{-1}}$,
\begin{equation*}
\left\|\BB(s\cdot) - r^{\nu h}\sqrt{\log \log r^{-1}} f\left(\frac{s}{r}\cdot\right)\right\|_\infty \geq \left\|\BB(s\cdot) - a f\right\|_\infty - b\left\|f-f\left(\frac{s}{r}\cdot\right) \right\|_\infty - (b-a)\|f\|_\infty 
\end{equation*}
and we find a bound for each of the last two terms (the first one is exactly the one given in the Lemma). We need the following inequality for $f\in H^\nu$, $s,t\in[0,1]^\nu$:
\begin{equation*}
|f(s)-f(t)|^2 \leq M_1 \|s-t\|^{2h} \|f\|_{\nu}^2 \ ,
\end{equation*}
which follows from approximation of $f$ by linear combinations of simple functions of the form $\lambda(\mathbf{1}_{[0,t_i]}\mathbf{1}_{[0,\cdot]})$ and the upper bound in Lemma \ref{lem:compDist} (where the constant $M_1$ comes from). Thus,
\begin{equation*}
- b \frac{(\log\log u^{-1})^{h/\nu}}{u^{\nu h}} \left\|f-f\left(\frac{s}{r}\cdot\right) \right\|_\infty \geq -M_1 \left(\log\log u^{-1}\right)^{h/\nu+1/2} \left(1-\frac{s}{u}\right)^h \|f\|_{\nu} \ .
\end{equation*}
For the last term, we use the fact that:
\begin{equation*}
b-a \leq \sqrt{u^{2\nu h} \log\log u^{-1} - s^{2\nu h}\log \log s^{-1}} \ ,
\end{equation*}
which ends the proof of this lemma.

\end{appendices}

\end{document}